\documentclass[12pt]{amsart}
\usepackage{amssymb,amsmath,epsfig,color,comment}
\usepackage[utf8]{inputenc}
\usepackage[letterpaper,total={6.5in,9in},top=1in, left=1in, right=1in, bottom=1in]{geometry}

\usepackage{fancyhdr}
\usepackage{graphicx}
\usepackage{stackrel}
\usepackage[abs]{overpic}
\usepackage[center]{caption}
\newtheorem{theorem}{Theorem}[section]
\newtheorem{lemma}[theorem]{Lemma}
\newtheorem{definition}[theorem]{Definition}

\newtheorem{proposition}[theorem]{Proposition}
\newtheorem{remark}[theorem]{Remark}

\newtheorem{corollary}[theorem]{Corollary}

\newtheorem{exercise}[theorem]{Exercise}

\newcommand{\ZZ}{\mathbb{Z}}

\title{Using Fibonacci Numbers and Chebyshev Polynomials to Express Fox Coloring Groups  and Alexander-Burau-Fox Modules of Diagrams of Wheel Graphs \\
}

\author{Anthony Christiana}
\address{Department of Mathematics, The George Washington University, Washington DC, USA}
\email{{\rm ajchristiana@gwmail.gwu.edu}}

\author{Huizheng Guo}
\address{Department of Mathematics, The George Washington University, Washington DC, USA}
\email{{\rm hguo30@gwmail.gwu.edu}}

\author{J\'{o}zef H. Przytycki}
\address{Department of Mathematics, The George Washington University, Washington DC, USA and \linebreak Department of Mathematics, University of Gda\'{n}sk, Gda\'{n}sk, Poland}
\email{{\rm przytyck@gwu.edu}}



\begin{document}

\begin{abstract}
In this paper we compute the Reduced Fox Coloring Group of the diagrams of Wheel Graphs which can also be represented as the closure of the braids $(\sigma_1 \sigma_2^{-1})^n$. In doing so, we utilize Fibonacci numbers and their properties.

Following this, we generalize our result to compute the Alexander-Burau-Fox Module over the ring $\mathbb{Z}[t^{\pm 1}]$ for the same class of links. In our computation, Chebyshev polynomials function as a generalization of Fibonacci Numbers.
\end{abstract}

\maketitle
\tableofcontents
\date{March 2023}

\bigskip
\section{Introduction}

The initial motivation for our work came from curiosity generated by Example  2.15 of \cite{Prz2} (see also Example 5.2 of   \cite{BGMMP}), 
where the determinant of the closure of the braid $(\sigma_1\sigma_2^{-1})^n$ is computed.
After successfully finding the group of Fox colorings of these links, we shifted our attention to the Alexander module of these links. In doing so the relation to Burau representation of braids becomes clear; thus we called our modules the Alexander-Burau-Fox modules. We use elementary methods, using properties of Fibonacci numbers and Chebyshev polynomials to obtain a simple closed formula for the group of Fox colorings and the ABF-module of $(\sigma_1\sigma_2^{-1})^n$. The connection to Plans' theorem on branched covers of $S^3$ along links added motivation  to our results.

In this section we introduce the Tait Diagram of a plane graph. We then recall the definition of Fox $n$-colorings and of the Fox Coloring Group of a link (the universal object for Fox Colorings). Then, we formulate our main result about the structure of Fox Coloring Groups.

In the second section, we prove our main result about Fox Colorings by working with the matrix of relations for the Fox Group using the properties of Fibonacci numbers.

In the third section, we recall the definition of the Alexander-Burau-Fox Module over the ring $\mathbb{Z}[t^{\pm 1}]$. This module is a generalization of our Fox Coloring Group. We describe the structure of the ABF Module for the same family of links considered in section one and two, and express it as the sum of two cyclic modules.

In the fourth section, we mention the relation with Plan's Theorem on branch covers of links. We also relate our results to those of Minkus and Mulazzani-Vesnin.

\bigskip

We now present some basic definitions. 

\begin{definition}\label{Tait Graph}
    For a plane graph $G,$ we associate an alternating link diagram as follows: 

    \begin{enumerate}
        \item Every edge is replaced by a crossing as illustrated in Figure \ref{taitcross}.
        \item  We connect the ``loose endpoints" of the crossing along the edges; compare Figure \ref{W7}.
    \end{enumerate}
\end{definition}

\begin{figure}[h]
    \centering
    \includegraphics[scale=1]{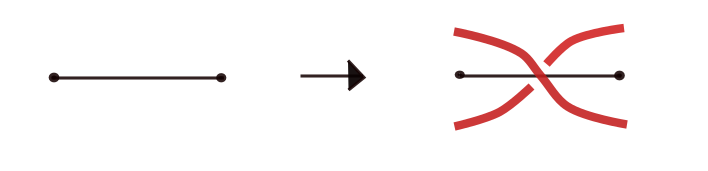}
    \caption{Crossing from an edge}
    \label{taitcross}
\end{figure}

\begin{definition}\label{Fox n-colorings}
    We define a \textit{Fox $n$-coloring} of a diagram $D$ to be a function
\[f: arcs(D) \to \mathbb{Z}_n \] such that every arc is ``colored" by an element of $\mathbb{Z}_n$ with the following condition: for every crossing with arcs $a,b,$ and $c$, $2b-a-c\equiv 0 \mod{n}$ for overarc $b$. That is, each crossing has the relation

\begin{figure}[h]
    \centering
    \includegraphics[scale=.25]{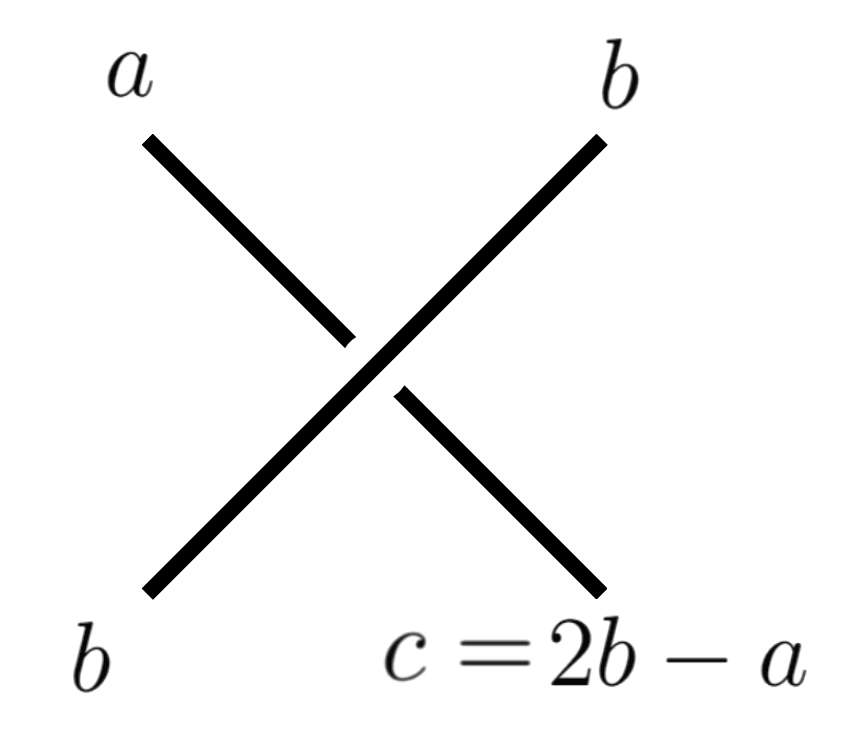}
    \caption{Coloring of a crossing}
    \label{CrossingColoring}
\end{figure}

The $n$-colorings of a diagram $D$ form a group denoted by  $Col_n(D)$. 

If $f(a_i) = f(a_j)$ for all $a_i, a_j \in arcs(D),$ we call $f$ a \textit{trivial coloring.} These trivial colorings  form the subgroup  $Col_n^{trivial}(D) \cong \mathbb{Z}_n \in Col_n(D)$ and the quotient $Col_n(D)/Col_n^{trivial}(D)$ is called the Reduced Group of Fox $n$-colorings denoted by $Col_n^{red}(D).$

\end{definition}

\begin{definition}
    An arc is the part of a diagram from undercrossing to undercrossing. We also include in our definition components without a crossing. 
\end{definition}

The number of arcs is equal to the number of crossing plus the number of trivial components of the diagram.

\begin{definition}\label{1.4}
    The group  $Col(D)$ is the abelian group whose generators are indexed by the arcs of $D$. The set of arcs is denoted by $arcs(D)$ and the set of generators is denoted by $Arcs(D)$. The relations at each crossing of $D$ are given by
    $2b-a-c=0$ where $a,b,c \in Arcs(D)$. That is, $$Col(D) = \displaystyle \lbrace \text {Arcs}(D) \ | \ \ \vcenter{\hbox{
\begin{overpic}[scale = .08]{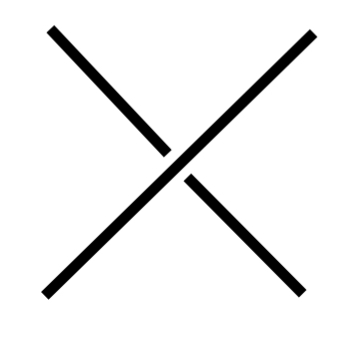}
\put(26, 26){\tiny{$b$}}
\put(0, -1){\tiny{$b$}}
\put(20, -1){\tiny{$c=2b-a$}}
\put(0, 26){\tiny{$a$}}

\end{overpic}}}\  \  \ \ \ \ , \ \ \text{ where $a,b,c\in Arcs(D)$}\rbrace $$

\end{definition}
The Fox Coloring Group can be also defined for tangles in a similar way. In particular for braids(for n braids treated as n-tangles), the group $Col(T)$ is freely generated by the top arcs of the braid.
We will use this later in the proof of Theorem \ref{mainFib}.

\begin{definition}\label{coltrivialdefi}
    Let $Col^{trivial}(D) \leq Col(D)$ be the infinite cyclic subgroup generated by the element $\sum_{a_i \in Arcs(D)} a_i$. This subgroup is isomorphic to $\mathbb{Z}$ and is called the group of trivial colorings of $D$.
   
    The  quotient group $\displaystyle \frac{Col(D)}{Col^{trivial}(D)}$ is called the \textbf{reduced group of Fox colorings}.\footnote{The group $Col^{red}(D)$ can be interpreted as the first homology of the double branch covering of $S^3$ branched along $D$; see \cite{Prz1}. The group $Col^{red}(D)$ can also be computed by the Goeritz matrix of the diagram. Using this approach, the determinant of $D_n$ is computed in \cite{Prz2}. } We denote it by $Col^{red}(D)$. That is, $$Col^{red}(D) = \displaystyle \lbrace \text {Arcs}(D) \ | \ \ \vcenter{\hbox{
\begin{overpic}[scale = .08]{CrossingMatrixRelation.jpg}
\put(26, 26){\tiny{$b$}}
\put(0, -1){\tiny{$b$}}
\put(20, -1){\tiny{$c=2b-a$}}
\put(0, 26){\tiny{$a$}}

\end{overpic} }} 
\  \  \ \ \ \ , \ \ \sum_{a_i \in Arcs(D)} a_i = 0\rbrace,$$
where the sum is taken over all arcs of $D$.

\end{definition}

More information about Fox Colorings can be found in \cite{Prz1, PBIMW}.

We are now ready to formulate the main theorem of the second section, expressing the Reduced Fox Coloring Group using Fibonacci numbers.

\begin{theorem}\label{mainFib}Let $F_k$ be the Fibonacci sequence defined by \[ F_0=0, \ F_1=1 \ \text{and } F_{k+2}=F_{k+1}+F_k. \] Let $D_n$ be the closure of the braid $(\sigma_1\sigma_2^{-1})^n,$ that is, $D_n=D(W_n)$ as in Figure \ref{W7}. Then 
\[
Col^{red}(D_n)=
\left\{
\begin{array}{ll}
\mathbb Z_{F_{n-1}+F_{n+1}} \oplus \mathbb Z_{F_{n-1}+F_{n+1}} & \mbox{when $n$ is odd},\\
\mathbb Z_{5F_{n}} \oplus \mathbb Z_{F_{n}} & \mbox{when $n$ is even}
\end{array} 
\right.
\]

In particular, for $n=2, 3, 4, 5, 6, 7,$ we have $\ZZ_5$, $\ZZ_4 \oplus \ZZ_4$, $\ZZ_{15} \oplus \ZZ_3$, $\mathbb Z_{11} \oplus \mathbb Z_{11}$, $\mathbb Z_{40} \oplus \mathbb Z_{8}$, and $\mathbb Z_{29} \oplus \mathbb Z_{29}$, respectively.

\end{theorem}

\vspace{0.5cm}

\section{Fibonacci numbers and the Reduced Coloring Group of Wheel Graphs, $Col^{red}(D)$}

After a few preliminaries, we turn to a proof of Theorem \ref{mainFib}.

The Tait diagram of a planar graph has a  determinant equal to the number of spanning trees of the graph. For wheel graphs in particular, the number of spanning trees was computed in \cite{Sed,Mye}. For our purposes, the determinant is equal to the order of the Reduced Fox Coloring Group. If the determinant is zero, which may only occur in the case of links, then it may result in infinite order. 

In the second section of the paper, we obtain concise formulas, using Fibonacci numbers, for the Fox-Coloring Group of diagrams obtained from wheel graphs, which can also be expressed as the closure of 3-braids of the form $(\sigma_1 \sigma_2^{-1})^n$  (compare \cite{DPS}) where $n$ is the number of spokes in the corresponding wheel graph; compare Figure \ref{W7}.

\begin{figure}[h]
    \centering
\includegraphics[scale=.8]{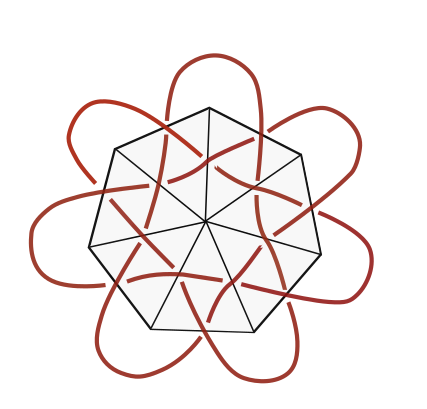}
\caption{Wheel graph  $W_7$ and its Tait diagram, $D_7=D(W_7)$ representing the closure of the braid $(\sigma_1 \sigma_2^{-1})^7.$}
\label{W7}
\end{figure}

The proof is comprised of several propositions and lemmas. 
\begin{proposition}
    Consider an arbitrary 3-braid $B$ and label the top arcs by $a$, $b$ and  $c$ and the bottom arcs by $a'$, $b'$, and $c'$, which are uniquely defined by $a$, $b$ and  $c$. \begin{enumerate}
        \item In the Fox Coloring Group of this 3-braid, our arcs satisfy the equation
$(a'-a) - (b'-b) +(c'-c)=0$.
\item Recall that $a$, $b$ and  $c$ form a basis of $Col(B)$ and we can change the basis to $b$, $b-a$ and $b-c$. Then $a'-a$, $b'-b$, and $c'-c$ are linear combinations of $b-a$ and $b-c$.  
    \end{enumerate}
\end{proposition}

\begin{proof}
  The proposition is well-known and can be proven by the induction on the number of crossings, see e.g \cite{DJP}.
\end{proof}

To simplify notation, let 
$x=b-a$ and $y=b-c$. Let $\hat{B}$ be the closure of $B$, then $Col(\hat{B}) = \{ b,x,y \ | \ a'-a, c'-c\}$. Notice that by Proposition 2.1, $a'-a$ and $c'-c$ is a linear combination of $x$ and $y$.\footnote{The fact that $x$ and $y$ generate $Col^{red}$ can be easily proved by linear induction. There is a general fact, see e.g. \cite{DJP} that if we consider a general $n$-tangle, $T$ with boundary points denoted by $x_1,x_2,...,x_{2n}$ (variable indexed by endpoints) then the relation $\sum_{i=1}^{2n}(-1)^ix_i= 0$ holds in $Col(T)$. In particular elements $x_i-x_{i+1}$ generate $Col^{red}(T)$ group.  } Let $a'-a =P_B=P_B(x,y)=P_B^x\cdot x +P_B^y \cdot y$ and $c'-c=Q_B=Q_B(x,y)= Q_B^x\cdot x +Q_B^y \cdot y$. 

The matrix of relations for $Col^{red}(\hat{B})$ is a $2 \times 2$ matrix 
\[
\left[
\begin{array}{ll}
P^x_{B} & P^y_B \\
Q^x_B & Q^y_B
\end{array}
\right].
\]

Now we work with $B = (\sigma_1\sigma_2^{-1})^n$ and $\hat{B} = D(W_n)$. Therefore, we use notation  $P_B = P_n$ and $Q_B = Q_n$. 

The following lemmas are illustrated by Figure \ref{S1S2inverse}.

\begin{figure}[h]
    \centering
\includegraphics[scale=.3]{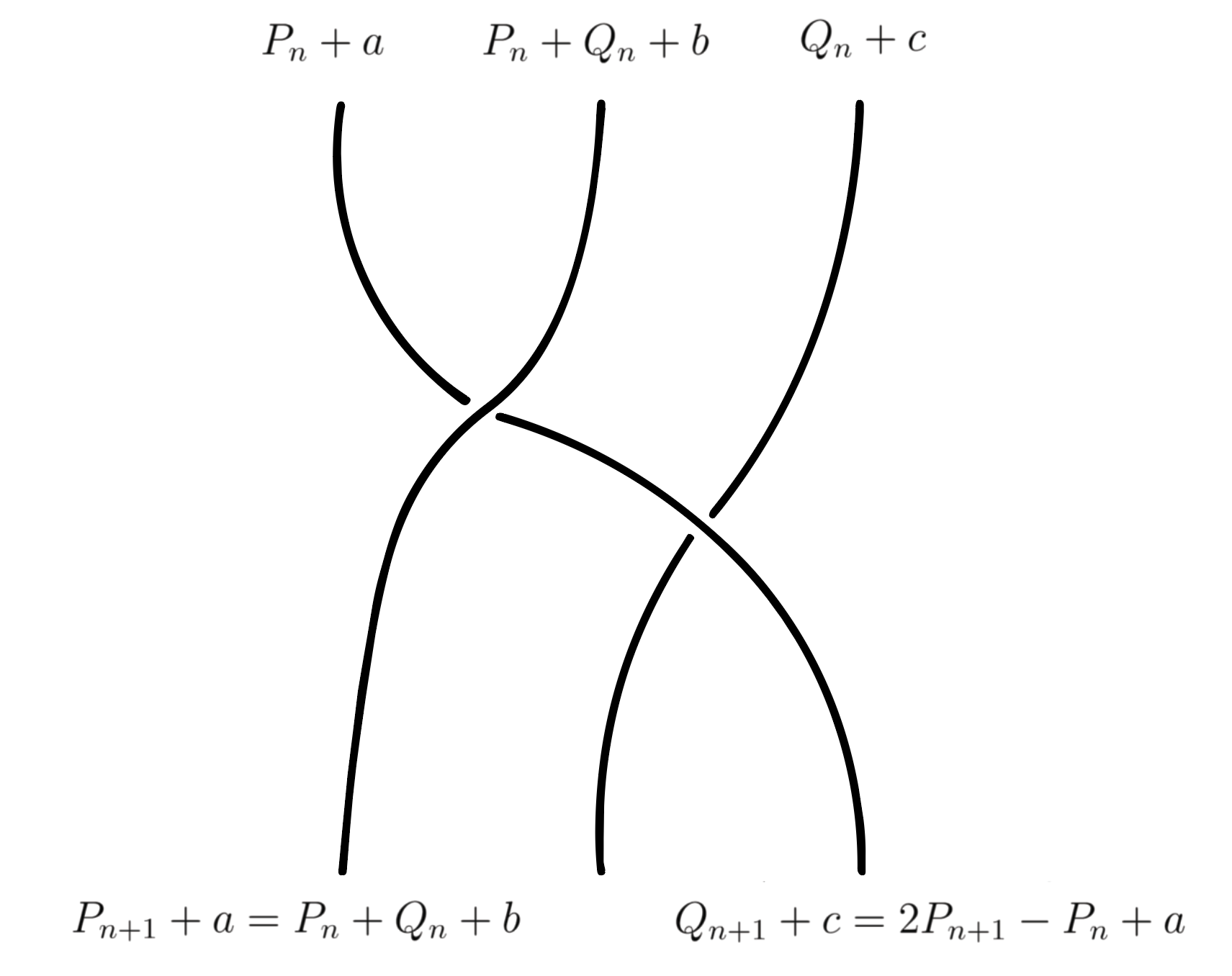}
\caption{Labeling part of $(\sigma_1\sigma_2^{-1})^n$}
\label{S1S2inverse}
\end{figure}

\begin{lemma}\label{recursion-fox} 
\begin{enumerate}
\item [(1)] $P_{n+1}= P_n+ Q_n + b-a= P_n+ Q_n +x$; thus $Q_n=P_{n+1}-P_n-x$.
\item [(2)] $Q_{n+1}= 2P_{n+1} -P_n + a - c= 2P_{n+1} -P_n +y-x \stackrel{(1)} {=} P_{n+1}+Q_n +y$; thus $P_{n+1}=Q_{n+1}-Q_n-y$.
\end{enumerate} 
\end{lemma}

Thus we can deduce recursive formulas for $P_n$ and $Q_n$.
\begin{lemma}
\begin{enumerate}
\item[(Rec1)] $P_{n+2} \stackrel{(1)} {=} P_{n+1} + Q_{n+1}+x \stackrel{(2)}{=} 3P_{n+1}-P_n+y$.
\item[(Rec2)] $Q_{n+2} \stackrel{(2)} {=} P_{n+2} + Q_{n+1}+y \stackrel{(1)}{=}P_{n+1} +2Q_{n+1} +x+y \stackrel{(2)}{=} 3Q_{n+1}-Q_n+x.$
\end{enumerate}
\end{lemma}

\bigskip

Then using the notation $P_n=P^x_n \cdot x + P^y_n \cdot y \text{ and } Q_n=Q^x_n \cdot x + Q^y_n \cdot y $ we recognize $P_n^x$ and $Q_n^y$ as Chebyshev polynomials as follows:

\begin{lemma}
    \begin{enumerate}
\item[(I)] $P^x_{n+2}= 3P^x_{n+1}- P^x_{n}$ with $P^x_0=0$, $P^x_1=1$, $P^x_2=3$, ...\\ 
Thus $P^x_{n+1}= S_n(3)$ where $S_n(z)$ denotes the Chebyshev polynomial of the second kind. That is $S_0(z)=1, S_1(z)=z$, and 
$S_{n+2}(z)=zS_{n+1}(z)-S_n(z).$ Compare Section \ref{cheb3.1}.
\hspace{10cm}
$Q^y_{n+2} = 3Q^y_{n+1} - Q^y_n$ with  $Q^y_0=0$, $Q^y_1=1$, $Q^y_2=3$, ...\\
Therefore, $Q^y_{n+1}= P^x_{n+1}= S_n(3)$.

\vspace{0.1cm}
\item[(II)] $P^y_{n+2}=3P^y_{n+1}- P^y_n +1$ with $P^y_0=0$, $P^y_1=0$, $P^y_2=1$, ...\\
Similarly, $Q^x_{n+2}=3Q^x_{n+1}- Q^x_n +1$ with $Q^x_0=0$, $Q^x_1=0$, $Q^x_2=1$, ...\\
		Therefore we have $P^y_{n+1}=Q^x_{n}= P^x_{n+1}-P^x_n-1= S_n(3)- S_{n-1}(3)-1.$
\end{enumerate}
\end{lemma}

Let $F_k$ be the Fibonacci sequence defined by $ F_0=0, \ F_1=1 \ \text{and } F_{k+2}=F_{k+1}+F_k$. It relates to Chebyshev polynomial as follows:\

\begin{lemma} We have $F_{2n}=S_{n-1}(3)$.
\end{lemma}
\begin{proof} $F_{2n}=F_{2n-1}+F_{2n-2}= 2F_{2n-2}+F_{2n-3}= 3F_{2n-2}-F_{2n-4}.$ Thus $F_{2n}$ and $S_{n-1}(3)$ satisfy 
	the same recursive relation. Noting that $F_2=1=S_0(3)$ and $F_4=3=S_1(3)$ we conclude the lemma.
\end{proof}

The main result of this section can be expressed using Fibonacci number as follows:

\begin{theorem}\label{MT} Let $D_n$ be the closure of the braid $(\sigma_1\sigma_2^{-1})^n$ 

Then:
\[
Col^{red}(D_n)=
\left\{
\begin{array}{ll}
\mathbb Z_{F_{n-1}+F_{n+1}} \oplus \mathbb Z_{F_{n-1}+F_{n+1}} & \mbox{when $n$ is odd},\\
\mathbb Z_{5F_{n}} \oplus \mathbb Z_{F_{n}} & \mbox{when $n$ is even}
\end{array} 
\right.
\]
\end{theorem}

The sums $F_{n-1}+F_{n+1}$ are called Lucas numbers denoted by $L_n$. That is $L_0=2, L_1=1,L_2=3,L_3=4,L_4=7,...$ and $L_{n+2}=L_{n+1}+L_n$.

We divide the proof of Theorem \ref{MT} into three lemmas.

\begin{lemma}
The matrix of relations for $Col^{red}(D_n)$ is
\[
A_{2n}=
\left[
\begin{array}{ll}
F_{2n} & F_{2n-1}-1 \\
F_{2n+1}-1 & F_{2n}
\end{array}
\right].
\]   
\end{lemma}

\begin{proof}
   From Lemma 2.3, we know 
   $Col^{red}(D_n)$=
$\left[
\begin{array}{ll}
P^x_{n} & P^y_n \\
Q^x_n & Q^y_n 
\end{array}
\right]
 =
\left[
\begin{array}{ll}
P^x_{n} & P^x_{n}-P^x_{n-1} -1\\
P^x_{n+1}-P^x_n-1 & P^x_{n}
\end{array}
\right].$

Now by Lemma 2.4, $P^x_{n}-P^x_{n-1} -1 = S_{n-1}(3) - S_{n-2}(3) -1 = F_{2n} - F_{2n-2} -1= F_{2n-1} -1$. Therefore $P^x_{n+1}-P^x_n-1 = F_{2n+2} - F_{2n} -1 = F_{2n+1} -1$.

Hence, $Col^{red}(D_n)$ =
$\left[
\begin{array}{ll}
F_{2n} & F_{2n-1} -1\\
F_{2n+1} -1 & F_{2n}
\end{array}
\right].$ We denote this matrix $A_{2n}$.

\end{proof}

\begin{lemma}
$Col^{red}(D_n)=
\mathbb Z_{F_{n-1}+F_{n+1}} \oplus \mathbb Z_{F_{n-1}+F_{n+1}}$ when $n$ is odd.
\end{lemma}

\begin{proof}
\begin{enumerate} \item[(1)] We replace the first column in $A_{2n}$ by the difference of the first and the second column to get:

\[ A_{2n-1}=
\left[
\begin{array}{ll}
F_{2n-2}+1 & F_{2n-1}-1 \\
F_{2n-1}-1  & F_{2n} 
\end{array}
\right].
\]
\item[(2)] We replace the second column in $A_{2n-1}$ by the difference of the second column and the first column to get:
\[ A_{2n-2}=
\left[
\begin{array}{ll}
F_{2n-2}+1 & F_{2n-3}-2 \\
F_{2n-1}-1  & F_{2n-2}+1 
\end{array}
\right].
\]
\end{enumerate}
We generalize (1) and (2) to an arbitrary number of column operations (with the same pattern of operations): 

After performing $k$  operations of (1) and (2), we get

\[ A_{2n-2k}=
\left[
\begin{array}{ll}
F_{2n-2k}+F_{2k} & F_{2n-(2k+1)}-F_{2k+1} \\
F_{2n-(2k-1)}-F_{(2k-1)}  & F_{2n-2k} + F_{2k}
\end{array}
\right].
\]

Since $n$ is odd, let $k = \frac{n+1}{2}$; then we have:

\[ A_{n-1}=
\left[
\begin{array}{ll}
F_{n-1}+F_{n+1} & F_{n-2}-F_{n+2} \\
0  & F_{n-1}+F_{n+1}
\end{array}
\right]
\]

Hence, the gcd of $Col^{red}(D)$ (that is, the  gcd of entries of the matix) is $F_{n-1}+F_{n+1}$ and 
$det A_{n-1}= (F_{n-1}+F_{n+1})^2$. Therefore $Col^{red}(D_n) = \mathbb Z_{F_{n-1}+F_{n+1}} \oplus \mathbb Z_{F_{n-1}+F_{n+1}}$ for n  odd. \footnote{We use the standard fact that for $\mathbb{Z}$ modules given by 2x2 matrix\[
\left[\begin{array}{ll}a & b \\ 
c & d\end{array}
\right],
\] the related abelian group is equal to $\mathbb{Z}_{d/g} \oplus \mathbb{Z}_{g}$ where $d$ is the determinant of the matrix and $g$ is $gcd(a,b,c,d)$.}

\end{proof}

\begin{lemma}
$Col^{red}(D_n)=\mathbb Z_{5F_{n}} \oplus \mathbb Z_{F_{n}}$ when $n$ is even.
\end{lemma}

\begin{proof}
Let $k = \frac{n}{2}$ and use the proof of Lemma 2.8. 

Then we have:
$A_{n}=
\left[
\begin{array}{ll}
2F_{n} & F_{n-1}-F_{n+1} \\
F_{n+1}-F_{n-1} & 2F_{n}
\end{array}
\right]
= \left[
\begin{array}{ll}
2F_{n} & -F_n \\
F_n & 2F_{n}
\end{array}
\right] $

Adding twice the last column of the $A_n$ to the first, we get: $\left[
\begin{array}{ll}
0 & -F_n \\
5F_n & 2F_{n}
\end{array}
\right] $.

Here the gcd of $Col^{red}(D)$ is $F_n$ and $|Col^{red}(D)| = 5F_n^2$. Thus we have $Col^{red}(D_n) = \mathbb Z_{5F_{n}} \oplus \mathbb Z_{F_{n}}$ for $n$  even.

\end{proof}

\begin{remark}
In \cite{Prz2} the Goeritz matrix of $D(W_n)$ was reduced to 
\[
\left[
\begin{array}{ll}
S_{n-1}(3) &   1-S_n(3) \\
S_{n-2}(3)+1 & -S_{n-1}(3)
\end{array}
\right].
\]

To show that the group it presents is the same as before, we multiply the last column by $-1$ and use the equality $S_{n-1}(3)=F_{2n}$, 
to get:
\[
\left[
\begin{array}{ll}
F_{2n}     & F_{2n+2}-1 \\
F_{2n-2}+1 & F_{2n}
\end{array}
\right]
= \left[
\begin{array}{ll}
F_{2n}     & F_{2n+1}+F_{2n}-1 \\
F_{2n-2}+1 & F_{2n-1}+F_{2n-2}
\end{array}
\right]
\]
After a column operation we get the matrix
\[
\left[
\begin{array}{ll}
F_{2n}     & F_{2n+1}-1 \\
F_{2n-2}+1 & F_{2n-1}-1
\end{array}
\right]=
\left[
\begin{array}{ll}
F_{2n-1}+F_{2n-2}     & F_{2n}+F_{2n-1}-1 \\
F_{2n-2}+1 & F_{2n-1}-1
\end{array}
\right]
\]
After a row operation we get the matrix
\[
\left[
\begin{array}{ll}
F_{2n-1}-1     & F_{2n}-1 \\
F_{2n-2}+1 & F_{2n-1}-1
\end{array}
\right]
\]
which is the matrix considered in (1) above with a row exchanged. This matrix presents the same abelian group as $A_{2n}.$

\end{remark}

\newpage

\begin{remark}\label{Identities}
We can deduce several formulas from our column reductions.
For example $$F_{2n}=F_n(F_{n-1}+F_{n+1}).$$
	and 
 
 \[ F_{2n-1}-1= \begin{cases}
     F_n (F_{n-2}+F_n) & \text{for $n$ even} \\
     F_{n-1}(F_{n-1}+F_{n+1}) & \text{for $n$ odd} \\
 \end{cases} \]
\end{remark}

The identities in Corollary \ref{Identities} are well known; e.g. they follow from
identities in \cite{Koshy} page 97. In particular,  $F_{2n}=F_n(F_{n-1}+F_{n+1})$ is  exactly identity 54 of \cite{Koshy}.\footnote{There are also the following related identities: 
$$\mbox{Identity 52: } F_{2m+n}-(-1)^mF_n = F_mL_{m+n},$$
$$\mbox{Identity 53: } F_{2m+n}+(-1)^mF_n = F_{m+n}L_{m},$$ }
Thomas Koshy mentions that the identities follow directly from closed forms for $F_n$ and $L_n$. We present the proof for completeness. We consider roots of the polynomial $x^2-x-1=0$ say $\alpha$ and $\beta$, and note that  
$$F_n= \frac{\alpha^n -\beta^n}{\alpha - \beta} \mbox{ and } L_n=\alpha^n + \beta^n.$$
We have $\alpha=\frac{1+\sqrt 5}{2}$, $\beta=-\alpha^{-1}= \frac{1-\sqrt{5}}{2}$, $\alpha-\beta =\sqrt{5}$.
Our identities can be expressed as product-to-sum formulas (using $F_{-k}=(-1)^{k+1}F_k$):
\begin{proposition}
$$F_m(F_{n-1}+F_{n+1})= F_mL_n= F_{m+n} + (-1)^nF_{m-n}= F_{m+n}- (-1)^mF_{n-m}.$$
    
\end{proposition}

\begin{proof}
\begin{align*}
    F_mL_n & = \frac{(\alpha^m-\beta^m)(\alpha^n+\beta^n)}{\alpha-\beta} \\
    & = \frac{\alpha^{m+n}-\beta^{m+n} +\alpha^m\beta^n -\alpha^n\beta^n}{\alpha-\beta} \\
    & = \frac{\alpha^{m+n}-\beta^{m+n} + (-1)^n(\alpha^{m-n} -\beta^{m-n})}{\alpha-\beta} \\
    & = F_{m+n} +(-1)^nF_{m-n}. \\
\end{align*}

\end{proof}

\newpage
\section{Alexander-Burau-Fox matrix of a link diagram}

In this chapter we generalize our main result of the second section from the Fox group of colorings to the Alexander-Burau-Fox module over the ring $\mathbb{Z}[t^{\pm 1}]$. 
Where the Fibonacci numbers appeared in the computation of the the Fox Coloring Group, the analogous combinatorial object appearing in the computation of the ABF Module is the Chebyshev Polynomial of the second kind.
The modules we study are Alexander modules of links \cite{Ale} in a form which can be deduced from the Burau representation of braids \cite{Burau,Bur-Z}. These module directly generalize to the group of Fox colorings where we replace $-1$ by $t.$

The formal definition of the ABF Module follows (notice that for $t=-1$, we get Definition \ref{1.4}).

\begin{definition} Let $D$ be an oriented link diagram.
The  $\mathbb{Z}[t^{\pm 1}]$-module $ABF(D)$ has a presentation where generators are indexed by the arcs of $D$, denoted by $Arcs(D)$, and whose relations are given by crossings of $D$ (positive and negative) as follows

\begin{figure}[h]
    \centering
    \includegraphics[scale=.35]{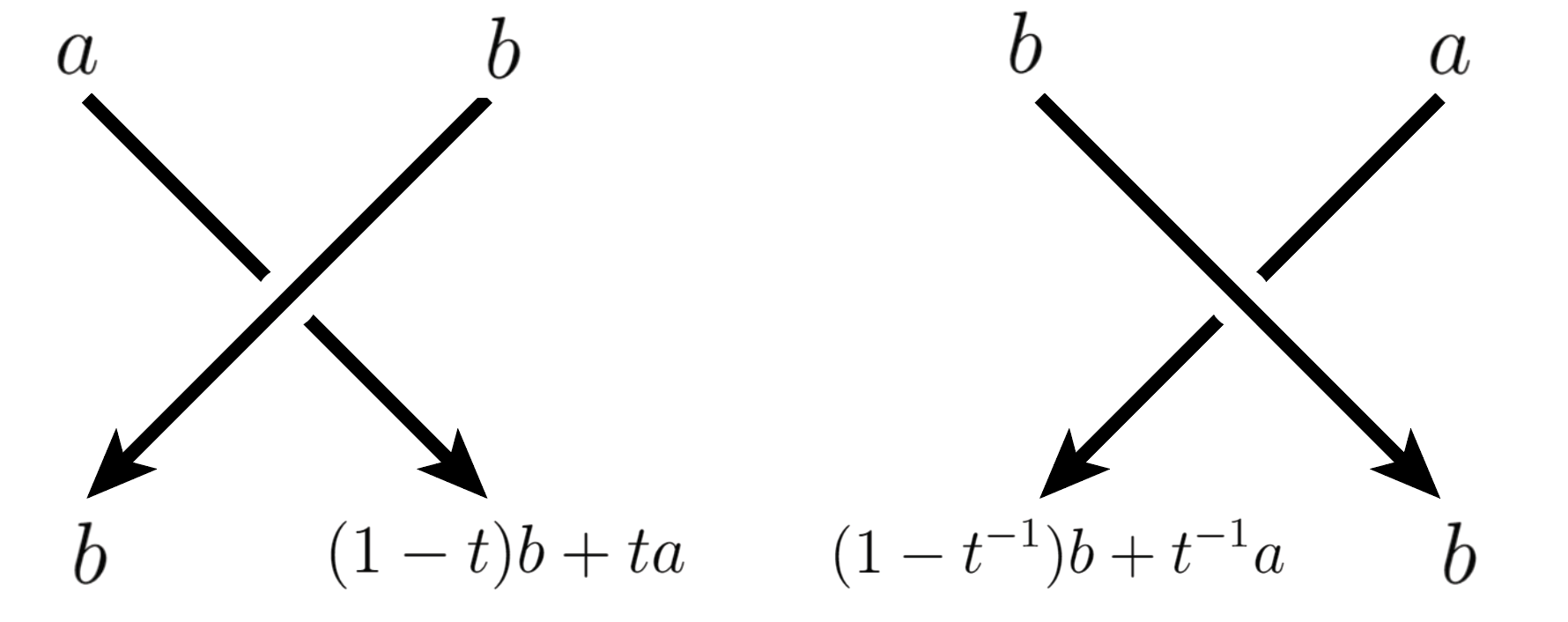} 
    \caption{ABF Module Relations}
    \label{AlRel}
\end{figure}

   That is, 
   \[ ABF_n(D) = \displaystyle \lbrace \text {Arcs}(D) \ | \ \ \includegraphics[scale=.35]{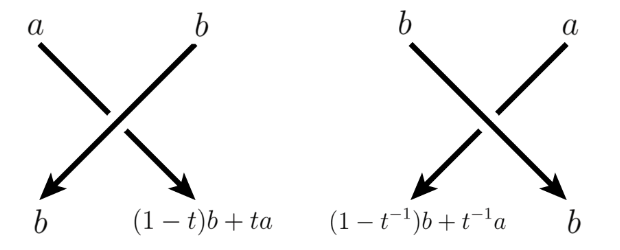} \text{ where $a,b,c\in Arcs(D)$}\rbrace \]
\end{definition}

\begin{definition}
    The Reduced Alexander-Burau-Fox Module is the module 
   \[ ABF_n^{Red}(D) =  \lbrace Arcs(D) \mid 
\includegraphics[scale=.15]{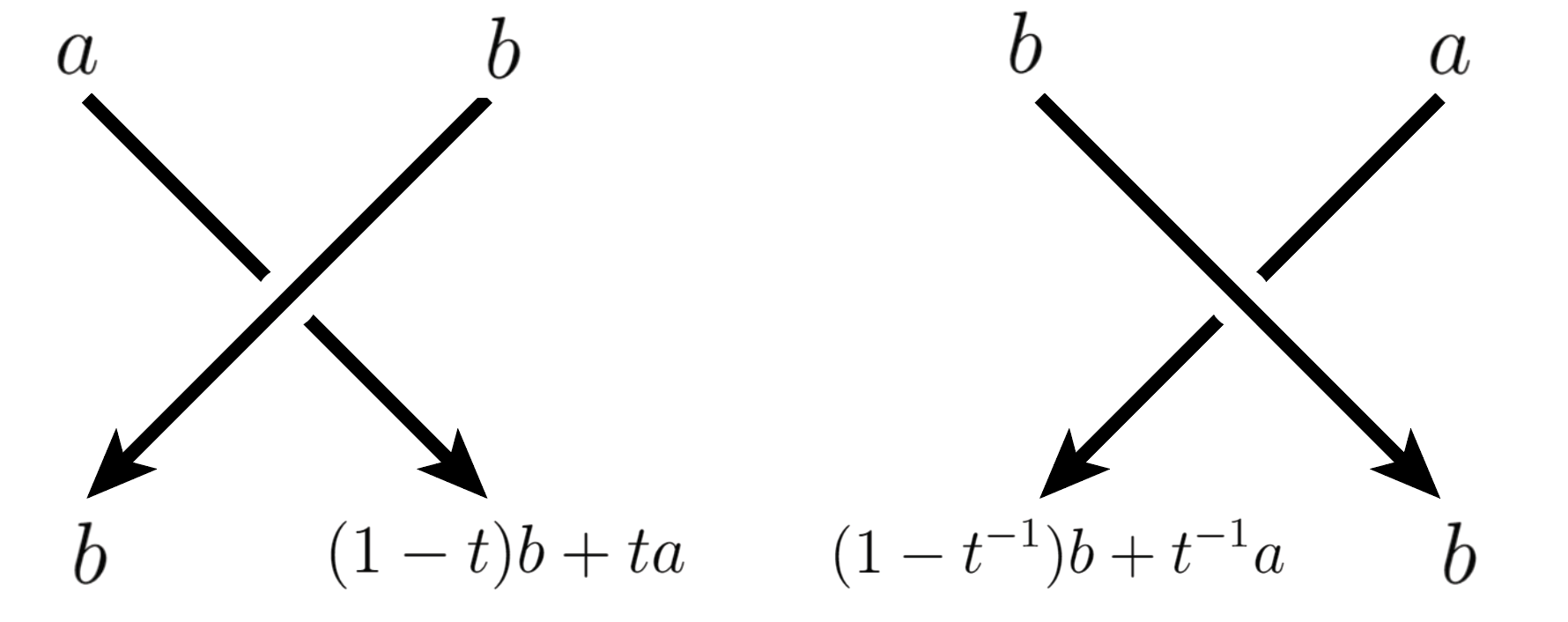}, \sum_{a_i \in arcs(D)} a_i = 0 \rbrace , \]

where the sum is taken over all arcs of $D$. Note that this module eliminates trivial colorings, which is equivalent to setting one arc, say $a_s$ equal to zero (removing one generator and adjusting relations by substituting $a_s= -\sum_{a_i,i\neq s}a_i$).\footnote{To be precise: the relation $\sum_{a_i \in arcs(D)} a_i = 0$ allows us to replace any other relation so that the coefficient of $a_s$ is equal to zero. In this new presentation $a_s$ can be removed and the presentation has generators $Arcs(D)- a_s$, and in relations we can put $a_s=0$. }

\end{definition}

We now formulate the main result of this section,
showing the structure of the ABF module of the diagrams of wheel graphs. We simplify our notation by calling $ABF(D(W_n)) = M_n(D_n).$ Notice that the module decomposes as the sum of cyclic modules (even though $\mathbb Z[t^{\pm 1}]$ is 
not a PID), and that for odd $n$ the module is double and for even $n$ it is ``almost" double. 

\begin{theorem}\label{maincheb}
    The reduced Alexander-Burau-Fox module of diagrams of wheel graphs $W_n$ can be expressed as 
\[
M_n^{Red}(D(W_n))=
\left\{
\begin{array}{ll}
\mathbb{Z}[t^{\pm 1}]/(S_{k-1}+S_k) \oplus \mathbb{Z}[t^{\pm 1}]/(S_{k-1}+S_k), & \mbox{when $n=2k+1$}\\
\mathbb{Z}[t^{\pm 1}]/(S_{k-1}) \oplus \mathbb{Z}[t^{\pm 1}]/((3-t-t^{-1})S_{k-1}) & \mbox{when $n=2k$}
\end{array} 
\right.
\]  where $z=1-t-t^{-1}$ and $S_k(z)$ is the Chebyshev polynomial of the second kind (see Subsection \ref{cheb3.1})
\end{theorem}

\bigskip

As before, we label the inputs of our braids $a,b,c$ and the outputs $a', b', c'$; see Figure \ref{Figure-3.2}.

\begin{figure}[h]
    \centering
    \includegraphics[scale=.35]{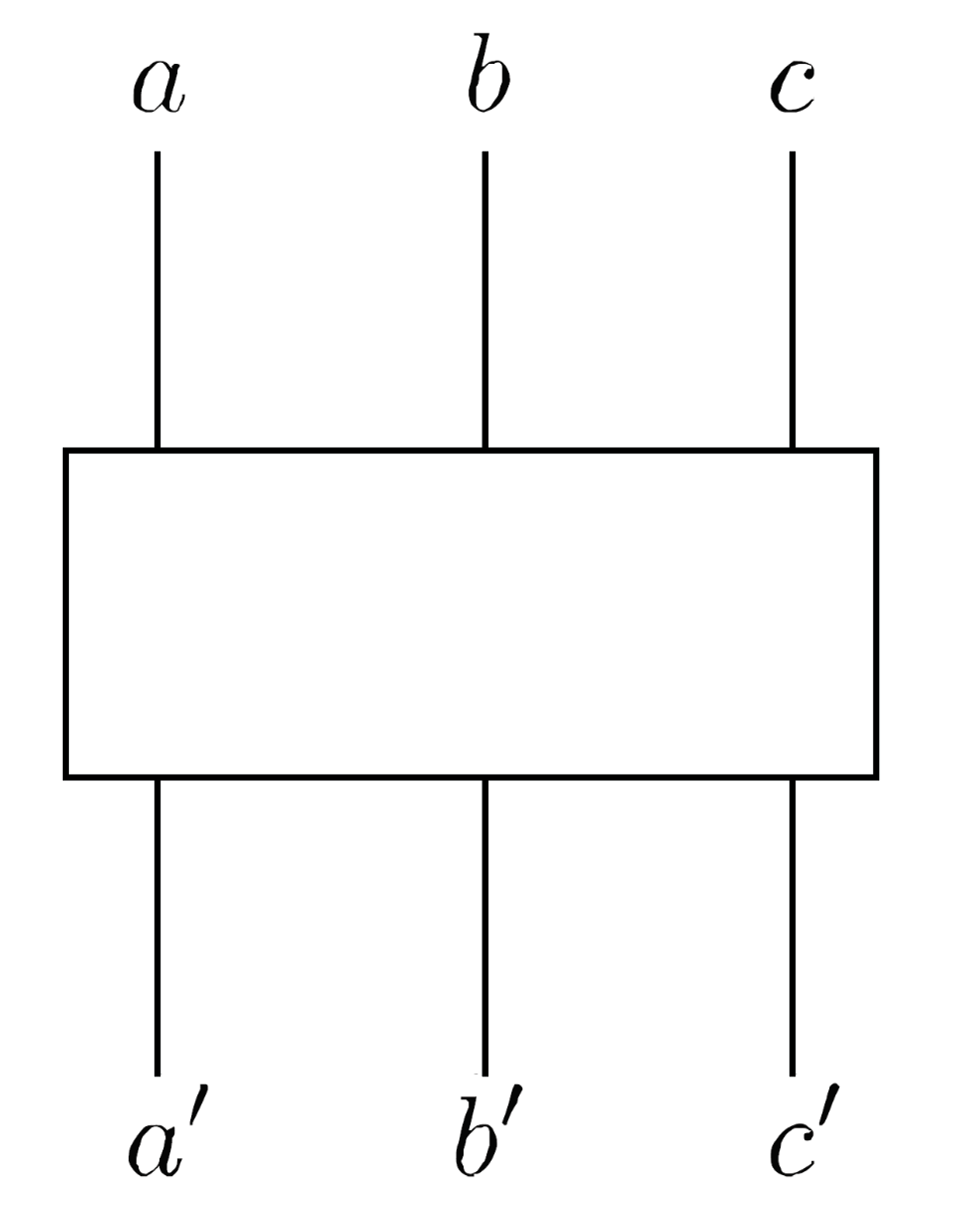}
    \caption{Coloring of the tangle}
    \label{Figure-3.2}
\end{figure}

As a special case of the general observation for n-tangles, we have that $t^2(a'-a)+t(b'-b)+(c'-c)=0$.  Thus $b'-b = -t(a'-a) -t^{-1}(c'-c)$, meaning that the middle color can be computed from the first and the third. 

Now like in the $t=-1$ case, let the endpoints of $(\sigma_1\sigma_2^{-1})^n$ be $a' = P_n+a$ on the left and 
$c' = Q_n+c$ on the right. Then $b'-b=-tP_n - t^{-1} Q_n$. As mentioned before, comoputing the Reduced ABF module is equivalent to setting $b = 0$. Thus, we have that $b' = -t P_n - t^{-1} Q_n.$

Writing $P_n= P_n^a \cdot a + P_n^c \cdot c$
and $Q_n= Q_n^a \cdot a + Q_n^c \cdot c$ as before, we see that the matrix of relations for the Reduced Alexander-Burau-Fox module, denoted $A_n$, is analogous to that of the Reduced Fox Coloring Group:

\[
A_n=
\left[
\begin{array}{ll}
P^a_{n} & P^c_n \\
Q^a_n & Q^c_n
\end{array}
\right].
\]

We also produce the following recursive relations, analogous to those found in Section 2. That is, 

\begin{lemma}\label{recursion0} 

    \begin{enumerate}
 \item[(1)] $P_{n+1}=-tP_n - t^{-1}Q_n -a$,     
\item[(2)] $Q_{n+1}= (1-t)P_{n+1} + tP_n +a-c$. 
\end{enumerate}

\begin{proof}
   Considering we work with the Reduced ABF Module, we set  $b=0$.  Let $b'_n$ denote the central strand after performing $(\sigma_1 \sigma_2^{-1})$ $n$ times. By previous observation, $b'_n = -tP_n - t^{-1} Q_n$. Then we have $P_{n+1}+a = b'_n$, so $P_{n+1} + a = b' = -tP_n - t^{-1} Q_n$. This proves (1) as is shown in Figure \ref{OneSegmentABF}.

\begin{figure}[h]
    \centering
    \includegraphics[scale=.35]{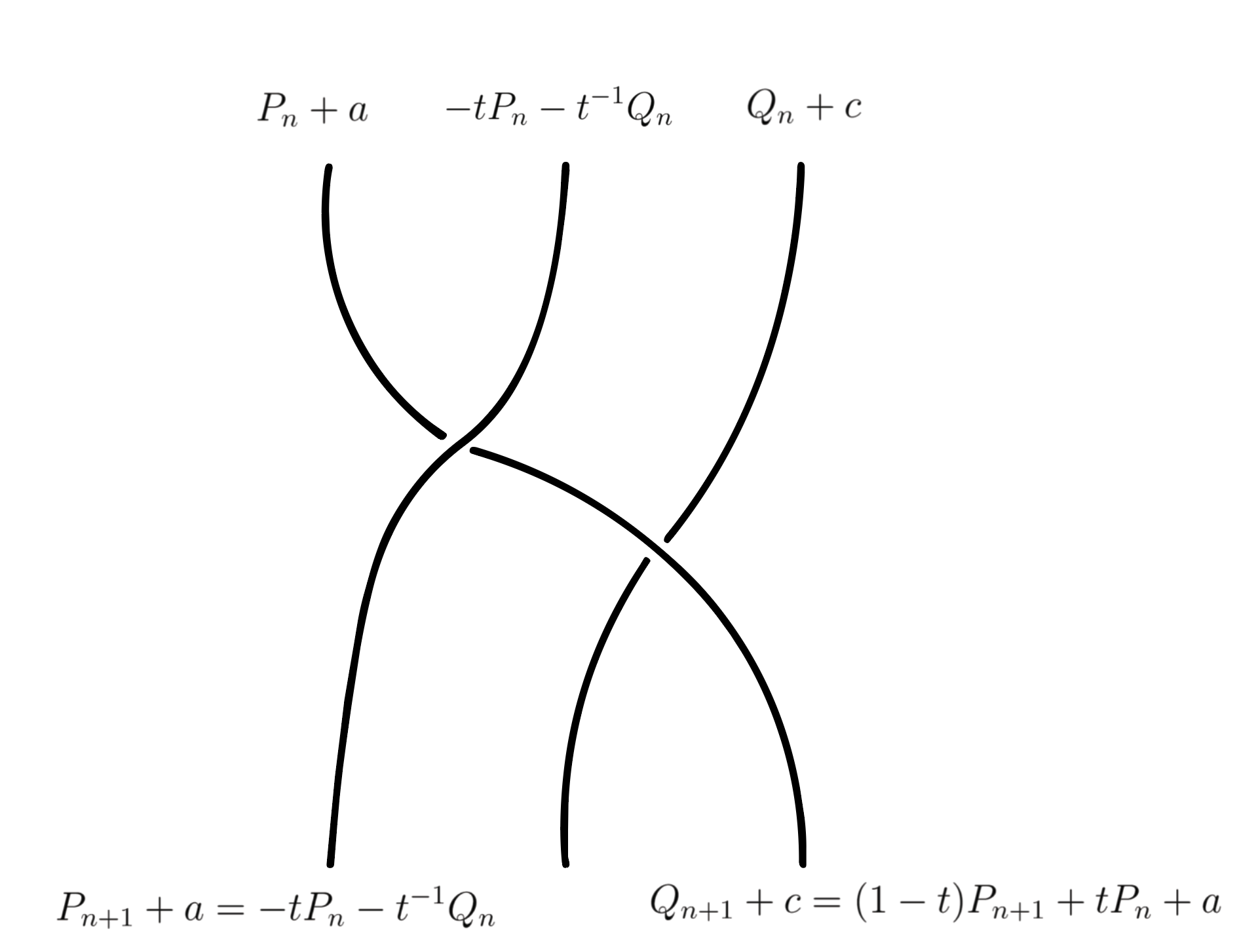}
    \caption{Labeling part of $(\sigma_1 \sigma_2^{-1})^n$ in ABF Module}
    \label{OneSegmentABF}
\end{figure}

    Next, we see that 
    \begin{align*}
        Q_{n+1} + c & = (1+t)(-t P_n - t^{-1} Q_n) + t(P_n + a) \\
        & = -t P_n - t^{-1} Q_n - t^2 P_n - Q_n + tP_n + a \\
        & = P_{n+1} - t(tP_n - t^{-1} Q_n) + t P_n + a \\
        & = P_{n+1} - tP_{n+1} + tP_n + a \\
      \implies  Q_{n+1} & = (1-t)P_{n+1} + t P_{n+a}. \\
    \end{align*}

    Thus, (2) is true.
\end{proof}

\end{lemma}

Immediately from the Lemma \ref{recursion}, we can present the matrix of relations $A_n$ in terms of $P_n^a$ and $P_n^c$ by the following lemma.

\begin{lemma}\label{QntoPn} For $Q_n^a$ and $Q_n^c$ as defined above,

 \begin{enumerate} 
        \item $Q_n^a = (1-t)P_n^a + t P_{n-1}^a + 1$
        \item $Q_n^c = (1-t)P_n^c + t P_{n-1}^c - 1$
    \end{enumerate} 
\end{lemma}
\begin{proof}
    By Lemma \ref{recursion0} (2), 
    \begin{align*}
        Q_n & = Q_n^a \cdot a + Q_n^c \cdot c \\
        & = (1-t)P_{n+1} + tP_n +a-c \\ 
        & = (1-t)P^a_{n}\cdot a + tP^a_{n-1}\cdot a +a  + (1-t)P^c_{n}\cdot c + tP^c_{n-1} \cdot c -c \\
        & = ((1-t)P^a_{n} + tP^a_{n-1} +1)\cdot a + ((1-t)P^c_{n} + tP^c_{n-1} -1) \cdot c. \\
    \end{align*}

\end{proof}

Thus, our Alexander-Burau-Fox matrix is 

\[ A_n=
\left[
\begin{array}{ll}
P^a_{n} & P^c_n \\
Q^a_n & Q^c_n
\end{array}
\right] = \left[
\begin{array}{ll}
P^a_{n} & P^c_n \\
(1-t)P_n^a + t P_{n-1}^a + 1 & (1-t)P_n^c + t P_{n-1}^c - 1
\end{array}
\right].\]
where $z = 1 -t - t^{-1}.$
\bigskip

The final identity we prove in service of our eventual translation of this matrix into the language of Chebyshev polynomials is the following:

\begin{lemma}\label{pnarec} For $P_n^a$ and $P_n^c$ we have the following Chebyshev-type recursive relation:
\begin{enumerate}
    \item $P_n^a  =z{P_{n-1}^a} - P_{n-2}^a - (1+t^{-1})$
    \item $P_n^c  = zP_{n-1}-P_{n-2}^c + t^{-1}$
\end{enumerate}
\end{lemma}
\begin{proof}
    Beginning with Lemma \ref{recursion} (1), 
    \begin{align*}
        P_n & = P_n^a \cdot a + P_n^c \cdot c \\
        & = -t P_{n-1} - t^{-1} Q_{n-1} - a \\
        & = -tP_{n-1} - t^{-1}((1-t)P_{n-1} + tP_{n-2} + a - c) -a \\
        & = -tP_{n-1} - t^{-1}P_{n-1} + P_{n-1} -P_{n-2} - t^{-1}a + t^{-1}c - a \\
        & = (1-t-t^{-1})P_{n-1} - P_{n-2} - t^{-1} a + t^{-1}c - a \\
        & = (zP_{n-1}^a - P_{n-2}^a - t^{-1} - 1)\cdot a - (zP_{n-1}^c - P_{n-2}^c + t^{-1}) \cdot c \\
    \end{align*}

    \end{proof}

\bigskip

\subsection{Chebyshev Polynomials}\label{cheb3.1}

We now set about rewriting $P_n^a$ and $P_n^c$  in terms of Chebyshev polynomials using Theorem \ref{recursion}, below.

We recall first the definition of Chebyshev Polynomials of the Second Kind.
\begin{definition}\label{chebdef}
   Let
\begin{align*}
    S_0(z) & = 1, \\
    S_1(z) & = z, \text{ and } \\
    S_n(z) & = z S_{n-1}(z) - S_{n-2}(z).
\end{align*}
Then $S_n(z)$ is the $n$-th polynomial of the Second Kind.  The definition of $S_{n}$ can be extended to negative $n$ by the same recursive relation. For example, $S_{-1}=0$.
 
Throughout the paper, we evaluate $S_n(z)$ at $z = 1 - t - t^{-1},$ though we tend to suppress the notation: $S_n(z) \to S_n.$
\end{definition}

\begin{theorem}\label{recursion} Let the sequence of polynomials $P_n(z,c_i)\in \mathbb Z[z,c_i]$ satisfy the recursive relation
$P_n = zP_{n-1}-P_{n-2} +c_i$. Then $P_n$ can be expressed using Chebyshev polynomials $S_n(z)$ as follows:
$$ P_n = S_{n-1}P_1 -S_{n-2}P_0 +\sum_{j=0}^{n-2}S_j c_{n-j}.$$
\end{theorem} 

\begin{proof}
    We proceed by induction on $n$. Let $n=3$ be our base case. Then \vspace{-.25 cm}
    \begin{align*}
        P_3 & = z P_2 - P_1 + c_3 \\ 
        & = z(zP_1 - P_0 + c_2) - P_1 + c_3 \\
        & = z^2P_1 - zP_0 + zc_2 - P_1 + c_3 \\
        & = (z^2 -1)P_1 - zP_0 + zc_2 + c_3 \\
        & = S_2 P_1 - S_1 P_0 + S_0 c_3 + S_1 c_2 \\
        & = S_2 P_1 - S_1 P_0 + \sum_{j=0}^1 S_j c_{n-j}. \\
    \end{align*} \vspace{-2 mm}
Assume that the proposition holds for $n \geq 3$ up to $n-1$. 

\bigskip

By inductive hypothesis, 
\begin{align*}
    P_n & = zP_{n-1} - P_{n-2} + c_n \\
    & = z(S_{n-2} P_1 - S_{n-3} P_0 + \sum_{j=1}^{n-3} S_j c_{n-1-j}) - (S_{n-3} P_1 -S_{n-4}P_0 + \sum_{j=0}^{n-4} S_j c_{n-2-j}) + c_n \\
    & = (zS_{n-2} - S_{n-3}) P_1 - (z S_{n-3} - S_{n-4})P_0 + z \sum_{j=0}^{n-3} S_j c_{n-1-j} - \sum_{j=0}^{n-4}S_j c_{n-2-j} + c_n \\
    & = S_{n-1} P -S_{n-2}P_0 + \sum_{j=0}^{n-3} S_j c_{n-1-j} - \sum_{j=0}^{n-4}S_j c_{n-2-j} + c_n \\
    & = S_{n-1} P -S_{n-2}P_0  + z S_0 c_{n-1} + \sum_{j=2}^{n-2}S_j c_{n-j} + c_n \\
    & = S_{n-1} P -S_{n-2}P_0   +S_0 c_n  + S_1 c_{n-1} + \sum_{j=2}^{n-2}S_j c_{n-j} \\
    & = S_{n-1} P -S_{n-2}P_0 + \sum_{j=0}^{n-2}S_j c_{n-j}.
\end{align*}

Thus, our identity holds.

\end{proof}

Thus we can rewrite the recursive relations in Lemma \ref{pnarec} with the Corollary \ref{tPcn}.
\begin{corollary}\label{tPcn}

\begin{enumerate}
    \item  $P_n^a(t)=- S_{n-1} -(1+t^{-1}) \sum_{j=0}^{n-2}S_j(z).$
    \item $P_n^c(t)=t^{-1} \sum_{j=0}^{n-2}S_j(z) \mbox{ for } z=1-t-t^{-1}.$
\end{enumerate}
\end{corollary}

\begin{proof} 
\begin{enumerate}
\item By Lemma \ref{pnarec} (1), $P_n^a  =z{P_{n-1}^a} - P_{n-2}^a - (1+t^{-1})$,  so applying Theorem \ref{recursion} with $c = -1 -t^{-1}$ yields the result with initial conditions $P_0^a = 0$ and $P_1^a = -1$.
    \item By \ref{pnarec} (2),  $P^c_n= zP^c_{n-1}-P^c_{n-2}+t^{-1}$ with initial conditions $P^c_0=0$ and $P^c_1=t^{-1}$.  \end{enumerate}

\end{proof}

In the next section, we simplify the equation in Corollary \ref{tPcn} using product-to-sum formulas (in fact, we go from sum to product).

\subsection{Product-to-Sum Formulas}\label{PtS-Chebyshev}
Seeking further transformation of $A_n$, we list the following identities for Chebyshev polynomials in Proposition \ref{chebident}, which are  analogous to those given for the Fibonacci numbers. They can be found in, for example \cite{Riv}, \cite{MH} or \cite{BIMPW}, but be aware that \cite{Riv,MH} use slightly different notation from ours, which follows the convention introduced to knot theory by R.Lickorish \cite{Lic}. Thus let $S_n$ be the Chebyshev polynomial of the second kind (see Definition \ref{chebdef}), and $T_n$ the Chebyshev polynomial of the first kind. That is, $T_0=1, T_1=x$ and $T_n=xT_{n-1}-T_n$. To obtain closed forms, substite $x=p+p^{-1}$, which yields
\[S_n = \frac{p^{n+1} - p^{-n-1}}{p-p^{-1}}\] and 
\[ T_n = p^n + p^{-n}. \]

These closed forms are helpful in the proofs of the five identities of Proposition \ref{chebident}, which are sketched for the convenience of the reader.

\begin{proposition}\label{chebident} For $S_n$ and $T_n$ as defined above,
\begin{enumerate}
\item[(1)] $T_mT_n = T_{m+n} + T_{m-n.}$\footnote{Some authors assume $m \geq n$ or take $|m-n|$ in the formula. It should be noted that this is not necessary as the formula works also for negative numbers.}
\item[(2)] $S_mT_n= S_{m+n} + S_{m-n}.$
\item[(3)] $S_mS_n= S_{m+n}+ S_{m+n-2}+...+S_{m-n+2}+ S_{m-n} = \sum_{i=m-n, i\equiv m+n \mod 2}^{m+n}S_i$ 
where the sum is taken over $i$ with $i\equiv m+n \mod 2$.
\item[(4)] $S_n(S_n+S_{n-1})= S_0+S_1+...+S_{2n-1}+ S_{2n} = \sum_{i=0}^{2n} S_i.$
\item[(5)] $S_n(S_n+S_{n+1})= S_0+S_1+...+S_{2n-1}+ S_{2n}+S_{2n+1}= \sum_{i=0}^{2n+1}S_i.$
\end{enumerate}
\end{proposition}
\begin{proof} 
(1): $T_mT_n = (p^m+p^{-m})(p^n+p^{-n})= p^{m+n}+p^{-m-n} + p^{m-n}+p^{n-m} = T_{m+n} + T_{m-n}$.\\ \ \\
(2):  $S_mT_n= \frac{(p^{m+1}-p^{-m-1})(p^n+p^{-n})}{p-p^{-1}} =\frac{(p^{m+n+1}-p^{-m-n-1}) + 
(p^{m-n+1} -p^{-m+n-1})}{p-p^{-1}}=S_{m+n} + S_{m-n}$.\\ \ \\
(3): We proceed by induction on $n$, applying (2) repeatedly:\\
	If $n=0$ we have $S_mS_0=S_m$ and for $m=1$, $S_mS_1=xS_m=S_{m+1}+S_{m-1}$. The inductive step holds as follows:
(assuming $m\geq n\geq 2$): $S_mS_n= S_m(S_n-S_{n-2})+S_mS_{n-2}=S_mT_n +S_mS_{n-2} \stackrel{(2)}{=} S_{m+n} + S_{m-n}+
S_mS_{n-2} \stackrel{ind}{=}S_{m+n} + S_{m-n}+ S_{m+n-2}+S_{m+n-4}+...+S_{m-n+2},$ as needed.\\
(4) We use (3) twice for $S_nS_n$ and for $S_nS_{n-1}$.\\
(5) We use (3) twice for $S_nS_n$ and for $S_nS_{n+1}$.

\end{proof}

The following two corollaries further simplify $A_n.$
\begin{corollary}\label{sumtoprod}
    \[
\sum_{j=0}^{n-2}S_j= \begin{cases}
S_{k}(S_{k}+S_{k-1}) & \mbox{when $n=2k$},\\
S_{k-1}(S_{k-1}+S_k)     & \mbox{when $n=2k+1$}. \\
\end{cases}
\]

\end{corollary}
\begin{proof}
    This is a restatement of (4) and (5) from Proposition \ref{chebident}.
\end{proof}

\begin{corollary}\label{2kstuff} \begin{enumerate}
    \item $S_{2k}= S_kS_k-S_{k-1}S_{k-1}=(S_k-S_{k-1})(S_k+S_{k-1})$
    \item $S_{2k+1}=S_kS_{k+1} -S_{k-1}S_k= S_k(S_{k+1}-S_{k-1}).$
\end{enumerate}
\end{corollary}
\begin{proof}
    \begin{enumerate}
        \item  By \ref{chebident} (3),
        
        \begin{align*}
        S_kS_k & = S_0+S_2+...+S_{2k} \\
        & = (S_0 + S_2 +\dots + S_{2k-2}) + S_{2k} \\
        & = S_{k-1} S_{k-1} + S_{2k} \\
        \end{align*}
        So $S_k S_k = S_{k-1} S_{k-1} + S_{2k}$ implies $S_{2k} = S_{k}S_{k} - S_{k-1} S_{k-1}.$
        \item $$S_kS_{k+1}=S_1+S_3+...+S_{2k+1},$$
    \end{enumerate}

\end{proof}

This allows us to reformulate $P_n^a$ and $P_n^c.$

\begin{lemma}\label{collected} \hspace{2 cm}
\begin{enumerate} 
\item $P^a_{2k} = -S_{k-1}((S_k+S_{k-1}) +t^{-1}(S_{k-1}+S_{k-2}))$
\item $P^a_{2k+1} = -(S_{k-1}+S_k)(S_k+t^{-1}S_{k-1})$
\item $P^c_{2k}= t^{-1} S_{k-1}(S_{k-1}+S_{k-2})$
\item $P^c_{2k+1}= t^{-1} S_{k-1}(S_{k-1}+S_k)$
\end{enumerate}
\end{lemma}

\begin{proof} \hspace{1 mm}

    \begin{enumerate}
        \item  By Corollary \ref{tPcn}, $P_{2k}^a(t)
       = - S_{2k-1}(z) -(1+t^{-1}) \sum_{j=0}^{2k-2}S_j(z)$, where $z = 1-t-t^{-1},$ as before. Thus,
        \begin{align*}
            P^a_{2k} & =- S_{2k-1} -(1+t^{-1}) \sum_{j=0}^{2(k-1)}S_j(z) \\
            & = -S_{2k-1} - (1+t^{-1})S_{k-1}(S_{k-1}+S_{k-2}) & \text{Cor. }\ref{sumtoprod} \\
            & = -S_{2(k-1)+1} - (1+t^{-1})S_{k-1}(S_{k-1}+S_{k-2}) \\
            & = - S_{k-1}(S_{k}-S_{k-2}) - (1+t^{-1})S_{k-1}(S_{k-1}+S_{k-2})  & \text{Cor. } \ref{2kstuff} \\ 
            & = -S_{k-1}(S_k - S_{k-2} +S_{k-1} + S_{k-2} + t^{-1}S_{k-1} + t^{-1} S_{k-2}) \\
            & = -S_{k-1}((S_k+S_{k-1}) -t^{-1}(S_{k-1}+S_{k-2})) \\
        \end{align*}
        \item By Corollary \ref{tPcn}, $P_{2k}^a(t)$
        \begin{align*}
            -P^a_{2k+1} & = -S_{2k+1-1} - (1+t^{-1})\sum_{j=0}^{2k+1-2} S_{j}(z) \\
            & = -S_{2k} - (1+t^{-1})S_{k-1}(S_{k-1} + S_{k+1}) & \text{Cor. \ref{sumtoprod}} \\
            & = -(S_k -S_{k-1})(S_k + S_{k-1}) - (1+t^{-1})S_{k-1}(S_{k-1} + S_k) & \text{Cor. } \ref{2kstuff} \\
            & = -(S_k + S_{k-1})(S_k - S_{k-1} + S_{k-1} + t^{-1} S_{k-1}) \\
            & = -(S_k + S_{k-1})(S_k + t^{-1}S_{k-1}) \\
        \end{align*}
        \item \& (4) These follow immediately from Cor \ref{sumtoprod}.
    \end{enumerate}
\end{proof}

Now, we present the Alexander-Burau-Fox matrix for closures of the braids of type $(\sigma_1 \sigma_2^{-1})^n$ in full generality.

\begin{proposition}\label{fullmatrix}

\[
A_n = \left[
\begin{array}{ll}
P^a_{n} & P^c_n \\
Q^a_n & Q^c_n
\end{array}
\right]=
\left[
\begin{array}{ll}
-g_n(g_{n+1}-t^{-1}g_{n-1}) & t^{-1}g_ng_{n-1} \\
tg_ng_{n+1} & -g_n(g_{n+1}-tg_{n-1})
\end{array}
\right]
\]
where 
\[
g_n= 
\left\{
\begin{array}{ll}
S_{k-1} & \mbox{when $n=2k$},\\
S_{k-1}+S_k & \mbox{when $n=2k+1$}.
\end{array}
\right.
\]
\end{proposition}

\begin{proof} This follows immediately from Lemma \ref{collected} and Lemma \ref{QntoPn}.
\end{proof}

\subsection{Alexander-Burau-Fox Module}

With the explicit presentation of the matrix $A_n$ in terms of Chebyshev polynomials from \ref{fullmatrix}, we compute the resulting module by row reduction.

Define 

\[ A'_n=A_n/(-g_n)=
\left[
\begin{array}{ll}
g_{n+1}-t^{-1}g_{n-1} & -t^{-1}g_{n-1} \\
-tg_{n+1} & g_{n+1}-tg_{n-1}
\end{array}
\right].
\]

One column operation, $\text{col}_1 - \text{col}_2 \to \text{col}_1$, yields 
\[ A'_n=
\left[
\begin{array}{ll}
g_{n+1} & -t^{-1}g_{n-1} \\
(-t-1)g_{n+1}+tg_{n-1} & g_{n+1}-tg_{n-1}
\end{array}
\right]
\]

We multiply the second column by $-t$:

\[ A'_n=
\left[
\begin{array}{ll}
g_{n+1} & g_{n-1} \\
(-t-1)g_{n+1}+tg_{n-1} & -tg_{n+1}+t^2g_{n-1}
\end{array}
\right]
\]

Specifically for $n=2k+1$, we get
\[A'_{2k+1}=
\left[
\begin{array}{ll}
S_k & S_{k-1} \\ -t S_k & -tS_k + t^2 S_{k-1} ]
\end{array}
\right].
\]

For $n=2k$, we get
\[A'_{2k}=
\left[
\begin{array}{ll}
S_{k-1} + S_k & S_{k-1} + S_{k-2} \\
-t(S_{k-1} + S_k) & -t(S_{k-1} + S_k) -t^2(S_{k-1} + S_{k-2}) 
\end{array}
\right].
\]

We will apply the following lemma to $A''_n$ in order to immediately compute the desired module.

\begin{lemma}\label{Lemma 25.2}
\begin{enumerate}
\item[(1)] The pair $(S_k,S_{k-1})$ can be reduced by Euclidean algorithm (column operations) to $(S_0,S_{-1})= (1,0).$
\item[(2)] The pair $(S_k+S_{k-1}, S_{k-1}+S_{k-2})$ can be reduced by Euclidean algorithm (column operations) to $(S_1+S_0, S_0+S_{-1})=(z+1,1)$.
\end{enumerate}
\end{lemma}
\begin{proof} (1) We notice that $(S_k,S_{k-1})= (zS_{k-1}-S_{k-2},S_{k-1})$ thus by taking the first column minus $z$ times the second 
we get $(-S_{k-2},S_{k-1})$. Multiplying the first column by $-1$ we get $(S_{k-2},S_{k-1})$ allowing the inductive step.\\
(2) We notice that $(S_k+S_{k-1}, S_{k-1}+S_{k-2})= (zS_{k-1}-S_{k-2}+zS_{k-2}-S_{k-3},S_{k-1}+S_{k-2})=
(z(S_{k-1}+S_{k-2})-(S_{k-2}+S_{k-3}), S_{k-1}+S_{k-2})$, thus by taking the first column minus $z$ times the second
we get $(-(S_{k-2}+S_{k-3}),S_{k-1}+S_{k-2})$ and multiplying the first column by $-1$ we get $(S_{k-2}+S_{k-3},S_{k-1}+S_{k-2})$,
allowing the inductive step.
\end{proof}

We apply  Lemma \ref{Lemma 25.2} to $A'_n$.
After applying this lemma we get the matrix
\[\left[
\begin{array}{ll}
1 & 0 \\
x & y\end{array}
\right] \mbox{ reduced by row operation to } 
\left[
\begin{array}{ll}
1 & 0 \\
0 & y\end{array}
\right]
\]
where $y= det A'_n$.\\
Thus we can conclude that our Alexander-Burau-Fox module given by $A_n$ is:
$$\mathbb Z[t^{\pm 1}]/(g_n)\oplus \mathbb Z[t^{\pm 1}]/(det A_n' \cdot g_n).$$

The $\det(A_n')$ is computed in the next section.

\subsection{Computing determinant of the matrix $A'_n$}
To complete the computation of the Alexander-Burau-Fox module of $(\sigma_1\sigma_2^{-1})^n$, we need the following result.

\begin{proposition} For $A'$
\begin{enumerate} \item $det A'_{2k+1}=1$.
\item $det A'_{2k}= 3-t-t^{-1}.$
\end{enumerate}
\end{proposition}
\begin{proof} (1) We have 
\begin{align*}
det A'_{2k+1} & = (S_k+t^{-1}S_{k-1})(S_k+tS_{k-1})- S_kS_{k-1} \\ 
    & =S_k^2+(t+t^{-1})S_kS_{k-1}+S_{k-1}^2-S_kS_{k-1} \\
    & =  S_k^2+(t+t^{-1}-1)S_kS_{k-1}+S_{k-1}^2= S_k^2 - zS_kS_{k-1}+S_{k-1}^2 \\
    & = S_k^2 -(S_{k+1}+S_{k-1})S_{k-1} +S_{k-1}^2 \\
    & = S_k^2- S_{k-1}S_{k+1} \\  
    & = 1 \\
\end{align*}\mbox{ by product-to-sum formulas}.
(2) We proceed as in the previous case. We have 
\begin{align*}
    det A'_{2k} &= \bigg((S_k+S_{k-1}) +t^{-1}(S_{k-1}+S_{k-2})\bigg)\bigg((S_k+S_{k-1})) +t(S_{k-1}+S_{k-2})\bigg) \\ & \hspace{1 cm}- 
(S_k+S_{k-1})(S_{k-1}+S_{k-2}) \\
	& = (S_k+S_{k-1})^2+(t+t^{-1})(S_k+S_{k-1})(S_{k-1}+S_{k-2}) +(S_{k-1}+S_{k-2})^2-(S_k+S_{k-1})(S_{k-1}+S_{k-2}) \\
    & = (S_k+S_{k-1})^2+(t+t^{-1}-1)(S_k+S_{k-2})(S_{k-1}+S_{k-2}) +(S_{k-1}+S_{k-2})^2 \\
    & = (S_k+S_{k-1})^2 - z(S_k+S_{k-1})(S_{k-1}+S_{k-2})+ (S_{k-1}+S_{k-2})^2 \\
    & = (S_k+S_{k-1})^2 -((S_{k+1}+S_{k-1})+(S_{k}+S_{k-2}))(S_{k-1}+S_{k-2}) +(S_{k-1}+S_{k-2})^2 \\
    & = (S_k+S_{k-1})^2 -(S_{k+1}+S_k)(S_{k-1}+S_{k-2}) \\
    & = (S_kS_{k-1}-S_{k+1}S_{k-3})+ (S_kS_k - S_{k+1}S_{k-1}) + (S_{k-1}S_{k-1} - S_kS_{k-2}) \\
    & =S_1+S_0+S_0 \\
    & = z+2 \\
    & = 1-t-t^{-1} +2  \\ 
    & = 3-t-t^{-1}.
\end{align*}
\end{proof}

Thus, with the computation of $\det(A_n'),$ we complete the proof of the main theorem, Theorem \ref{maincheb}.
\newpage

\section{Relation to Burau representation}
The crossing illustrated in Figure \ref{AlRel} can be interpreted as a 2-braid, and the ABF relation as a linear map: for a positive crossing, $B(a,b) = (b,ta+(1-t)b)$. This map is given in a 
standard basis $(1,0), (0,1)$ by the matrix:
\[
B=\left[ \begin{array}{cc}
0 & 1 \\
t & 1-t 
\end{array} 
\right]
\]
which is an element of $Gl(2,Z[t^{\pm 1}])$.
More generally, for an $n$-braid we have a homomorphism $B_n \to Gl(n,Z[t^{\pm 1}])$ given on the generators by 
\[
B(\sigma_i)=\left[ \begin{array}{cccc}
Id_{i-1} & 0 & 0 & 0 \\
0 & 0 & 1 & 0 \\
0 & t & 1-t & 0 \\
0 & 0 & 0 & Id_{n-i-1} 
\end{array} 
\right]
\]
This representation is known as the (unreduced) Burau representation \cite{Burau}. 
Because for a negative crossing we have 
 $B(b,a)=((1-t^{-1})b,t^{-1}a)$, therefore \[
B(\sigma_i^{-1})=\left[ 
\begin{array}{cccc}
Id_{i-1} & 0 & 0 & 0 \\
0 & 1-t^{-1} & t^{-1} & 0 \\
0 & 1 & 0 & 0 \\
0 & 0 & 0 & Id_{n-i-1} 
\end{array} 
\right].
\]

Clearly $B(\sigma_i)B(\sigma_i^{-1})=Id_n,$  
reflecting the second Reidemeister move.
In older literature the role of $\sigma_i$ and $\sigma_i^{-1}$ (and $t$ and $t^{-1}$) is changed (historically negative and positive was not always the same, and until now braid theorists use the opposite convention from knot theorists).

The matrix $B(\gamma)-Id_n$ where $\gamma\in B_n$ gives (with relations in rows) description of the unreduced 
Alexander-Burau-Fox module.  This is, in fact, the matrix used in our calculation in Chapter 3.

\section{Speculations and Future Directions}
To place our result in a broader context, we should mention the classical result of 
J.Minkus and Mulazzani-Vesnin \cite{Min,MV} relating $n$-fold branch coverings of 2-bridge links with double branch coverings of certain links. In relation to our family of braids, they notice that the $n$-fold branch covering of $S^3$ along 
the figure eight knot is homeomorphic to the double branch covering of $S^3$ branched along the closure of the braid $(\sigma_1\sigma_2^{-1})^n$. Thus our theorem on the group of Fox colorings can be reformulated in the language of homology of $M_{4_1}^{(n)}$, where 
$M^{(n)}_L$ denotes the $n$-fold branch coverings of $S^3$ branched along the link $L$.
 We plan to explore this connection to analyze left orderings of the fundamental groups of some branched coverings; compare \cite{DaPr}. We recall the results of \cite{BGW}, which show that the 2-fold branch covering of $S^3$ along a non-split alternating link produces a fundamental group which is not left-orderable.

A. Plans proved in \cite{Pla} that an odd branch covering of $S^3$ along any knot has a double form. That is, there exists $G$ such that 
\[ H_1(M_{k}^{(n)}) = G \oplus G \ \text{ for any odd $k$.}\]

V. Turaev produced a related result for $k$ even (compare \cite{Web,DW,Prz4}).  We observe a similar phenomenon in our computation of ABF for $(\sigma_1 \sigma_2^{-1})^n.$

Our theorem for the Reduced Fox Coloring Group corresponds well with Plans' and Turaev's results, but what does this mean in the generalization to the ABF module? 

Is there some reasonable interpretation of $(\sigma_1 \sigma_2^{-1})^\infty$ in relation to ABF module?

\section{Acknowledgments}
 The third author was partially supported by the Simons Collaboration Grant 637794.

\end{document}